\documentclass[12pt]{amsart}
\usepackage{latexsym}
\usepackage{amsthm}
\usepackage{amsmath}
\usepackage{amsfonts}
\usepackage{amssymb}
\usepackage[dvips]{graphicx}
\usepackage{xypic}
\input xy
\xyoption{all}

\newtheorem{theo}{Theorem}[section]
\newtheorem{lemma}[theo]{Lemma}

\newtheorem{defi}[theo]{Definition}
\newtheorem{coro}[theo]{Corollary}
\newtheorem{rem}[theo]{Remark}
\newtheorem{exam}[theo]{Example}

\newcommand\Po{\operatorname{Po}}
\newcommand\Tc{\operatorname{Tc}}
\newcommand\Rt{\operatorname{Rt}}
\newcommand\Ps{\operatorname{Ps}}

\newcommand\sm{\operatorname{sm}} 
\newcommand\psh{\operatorname{psh}}
\newcommand\smp{\operatorname{smp}}

\newcommand\Alg{\operatorname{Alg}}

\newcommand\Iso{\operatorname{Iso}}

\newcommand\id{\operatorname{id}}
\newcommand\Id{\operatorname{Id}}

\newcommand\C{\operatorname{\bf C}}
\newcommand\CAT{\operatorname{\bf CAT}}

\newcommand\COMB{\operatorname{\bf COMB}}
\newcommand\LOC{\operatorname{\bf LOC}}
\newcommand\CELL{\operatorname{\bf CELL}}
\newcommand\CMOD{\operatorname{\bf CMOD}}

\newcommand\Mono{\operatorname{Mono}}

\newcommand\Ins{\operatorname{Ins}}
\newcommand\Eq{\operatorname{Eq}}

\newcommand\cof{\operatorname{cof}}
\newcommand\tcof{\operatorname{tcof}}
\newcommand\cell{\operatorname{cell}}
\newcommand\colim{\operatorname{colim}}

\newcommand\cc{\mathcal {C}}
\newcommand\cd{\mathcal {D}}

\newcommand\ck{\mathcal {K}}
\newcommand\cl{\mathcal {L}}
\newcommand\cm{\mathcal {M}}

\newcommand\cs{\mathcal {S}}

\newcommand\cp{\mathcal {P}}

\newcommand\cx{\mathcal {X}}

\newcommand\Tcsh[1]{\operatorname{#1\textrm{-}Tc}}
 
\date{April 29, 2013}
 
\begin{document}
\title[Cellular categories]
{Cellular categories}
\author[M. Makkai and J. Rosick\'{y}]
{M. Makkai$^{*}$ and J. Rosick\'{y}$^{**}$}
\thanks{$^{*}$ Supported by the project CZ.1.07/4.00/20.0003 of the Operational Programme Education for Competitiveness of the Ministry
               of Education, Youth and Sports of the Czech Republic. $^{**}$ Supported by the Grant Agency of the Czech republic under the grant 
               P201/12/G028.} 
\address{
\newline 
Department of Mathematics and Statistics\newline
Masaryk University, Faculty of Sciences\newline
Kotl\'{a}\v{r}sk\'{a} 2, 611 37 Brno, Czech Republic\newline
makkai@math.mcgill.ca\newline
rosicky@math.muni.cz
}
 
\begin{abstract}
We study locally presentable categories equipped with a cofibrantly generated weak factorization system. Our main result is that these categories
are closed under $2$-limits, in particular under pseudopullbacks. We give applications to deconstructible classes in Grothendieck categories.
We discuss pseudopullbacks of combinatorial model categories.  
\end{abstract} 

\maketitle
 
\section{Introduction}
We introduce cellular categories as categories equipped with a class of morphisms containing all isomorphisms and closed under pushout and transfinite 
composite (= transfinite composition). The special case is a category equipped with a weak factorization system, which includes categories equipped 
with a factorization system. The latter categories are called ``structured " in \cite{AHS}. Cellular categories are abundant in homotopy theory because 
any Quillen model category carries two weak factorization systems, i.e., two cellular structures given by cofibrations and trivial cofibrations, resp. 
There are also various concepts of ``cofibration categories" equipped with cofibrations and weak equivalences (see \cite{RB} for a recent survey). One 
can do homotopy theory in any category equipped with a weak factorization system because we have cylinder objects and hence homotopies there (see 
\cite{KR}). Cellular category does not need to have weak factorizations -- for example pure monomorphisms in certain locally finitely presentable 
categories (see \cite{BR}). However, in a locally presentable category, one always has weak factorizations whenever cellular morphism are generated 
by a set of morphisms. The left part of the corresponding factorization system consists of retracts of cellular morphisms. In harmony with the J. Smith's 
concept of a combinatorial model category, we call such cellular categories combinatorial.

Our main result is that combinatorial cellular categories are closed under constructions of limit type. Like for locally presentable
(or accessible categories) these limits should be defined in the framework of $2$-categories and they can be reduced to products, inserters
and equifiers (see \cite{MP} and \cite{AR}). These limits are called PIE-limits. The consequence is that they include both lax limits and pseudolimits. 
It turns out that the key step is the closedness under pseudopullbacks and the key ingredience is the use of good colimits introduced by Lurie \cite{L}
and futher developed by the first author in \cite{M}. Lurie used good colimits for lifting cellular structure to functor categories,
which is a limit type construction.

Our starting point is  \cite{MRV} and we are using notation from that paper. Among others, the present paper links combinatorial cellular categories 
with deconstructible classes of objects in Grothendieck abelian categories (see \cite{G}, \cite{EAPT}, \cite{S}, \cite{SS} and \cite{Be}) where good 
colimits are replaced by generalized Hill lemma. Our limit theorem for combinatorial cellular categories implies some limit theorems
for deconstructible classes proved in \cite{S} and \cite{Be}.

\section{Combinatorial categories}
\begin{defi}\label{def2.1}
{\em  
A cocomplete category $\ck$ is called \textit{cellular} if it is equipped with a class $\cc$ of morphisms containing all isomorphisms and closed 
under pushout and transfinite composite.
}
\end{defi}

Morphisms belonging to $\cc$ are called \textit{cellular} and $\cc$ will be often denoted as $\cell(\ck)$. Given a class $\cx$ of morphisms
of a cocomplete category $\ck$ then $\cell(\cx)$ denotes the closure of $\cx$ under pushout and transfinite composite. In fact, $\cell(\cx)$
consists of transfinite composites of pushouts of morphisms from $\cx$. We say that the cellular category $(\ck,\cell(\cx))$ is \textit{cellularly
generated} by $\cx$. In a cellular category $\ck$, let 
$$
\cof(\ck)=\Rt\cell(\ck)
$$ 
consist of retracts of cellular morphisms in the category $\ck^2$ of morphisms of $\ck$. Elements of this class are called \textit{cofibrations}. 

\begin{lemma}\label{le2.2}
Let $\ck$ be a cellular category. Then $(\ck,\cof(\ck))$ is a cellular category.
\end{lemma}

\begin{proof}
It is easy to see that $\cof(\ck)$ is closed under pushout. Let $f_0:A_0\to A_1$ and $f_1:A_1\to A_2$ be two composable cofibrations. 
Following \cite{MRV} 2.1(5), there are cellular morphisms $g_0:A_0\to B_1$, $h_1:A_1\to C_2$ and morphisms $u_1:A_1\to B_1$, $r_1:B_1\to A_1$,
$v_2:A_2\to C_2$, $s_2:C_2\to A_2$ such that $u_1$, $r_1$ make $f_0$ a retract of $g_0$ in $A_0\backslash\ck$ and $v_2$, $s_2$ make $f_1$
a retract of $h_1$ in $A_1\backslash\ck$. Consider a pushout
$$
\xymatrix@=3pc{
B_1 \ar[r]^{g_1} & B_2 \\
A_1 \ar [u]^{u_1} \ar [r]_{f_1} &
A_2 \ar[u]_{u_2}
}
$$
Since $h_1r_1u_1=h_1=v_2f_1$, there is the unique morphism $t:B_2\to C_2$ such that $tg_1=h_1r_1$ and $tu_2=v_2$. It is easy to see that
$u_2$ and $s_2t$ make $f_1$ a retract of $g_1$. Thus $f_1f_0$ is a retract of $g_1g_0$. Consequently, cofibrations are closed under transfinite
composite.
\end{proof}

We say that a cellular category $\ck$ is \textit{retract closed} if $\cof(\ck)=\cell(\ck)$. Following \ref{le2.2}, $(\ck,\cof(\ck))$ is
a reflection of a cellular category $\ck$ into retract closed cellular categories. 

Given a set of morphisms in a cocomplete category $\ck$ then $\cof(\cx)$ denotes the closure of $\cx$ under pushout, transfinite composite
and retract. In fact, $\cof(\cx)$ consists of retracts of transfinite composites of pushouts of morphisms from $\cx$. We say that $\cof(\cx)$
is \textit{cofibrantly generated} by $\cx$. 

\begin{defi}\label{def2.3}
{\em  
A retract closed cellular category $\ck$ is called \textit{combinatorial} if $\ck$ is locally presentable and $\cof(\ck)$ is cofibrantly generated
by a set of morphisms from $\ck$.  
}
\end{defi}

In a combinatorial category $\ck$, the class of cofibrations forms a left part of a weak factorization system. Following \cite{L} A.1.5.12, 
any combinatorial category is cellularly generated by a set of morphisms. 

An object $K$ is called \textit{cofibrant} if the unique morphism $0\to K$ from an initial object is a cofibration. Analogously we define
\textit{cellular} objects.
 
\begin{exam}\label{ex2.4}
{
\em
Any cocomplete category carries two cellular structures -- the \textit{discrete} $\ck_d=(\ck,\Iso)$ and the \textit{trivial} $\ck_t=(\ck,\ck^2)$.
They are both retract closed. Moreover, if $\ck$ is locally presentable, they are both combinatorial. This is evident for the discrete one.
Let $\ck$ be a locally $\kappa$-presentable category. Then the trivial structure is combinatorial because $\ck^2$ is cofibrantly generated 
by morphisms between $\kappa$-presentable objects. (see \cite{R}, 4.6 for $\kappa=\omega$ but the general case is the same). 
}
\end{exam}

A morphism $(u,v):g\to f$ in $\ck^2$ will be called a \textit{pushout morphism} if the square
$$
\xymatrix@=3pc{
A \ar[r]^{f} & B \\
C \ar [u]^{u} \ar [r]_{g} &
D \ar[u]_{v}
}
$$
is a pushouts. Let $\psh(\ck^2)$ be the subcategory of $\ck^2$ with the objects as $\ck^2$ and with the pushout morphisms.

\begin{lemma}\label{le2.5}
Let $\ck$ be a locally $\kappa$-presentable category and $\cx$ a set of morphisms between $\kappa$-presentable objects. 
Then the full subcategory of $\psh(\ck^2)$ on objects belonging to $\Po(\cx)$ is locally $\kappa$-presentable 
with $\Po_\kappa(\cx)$ being the full subcategory of $\kappa$-presentable objects.
\end{lemma}
\begin{proof}
Our category is clearly cocomplete with colimits calculated in $\ck^2$ and objects from $\Po_\kappa(\cx)$ are 
$\kappa$-presentable. Thus it suffices to prove that any morphism in $\Po(\cx)$ is a $\kappa$-directed colimit 
in $\psh(\ck^2)$ of morphisms belonging to $\Po_\kappa(\cx))$.

Let $f:A\to B$ be a morphism in $\Po(\cx)$. Thus there is a pushout
$$
\xymatrix@=3pc{
A \ar[r]^{f} & B \\
X \ar [u]^{u} \ar [r]_{g} &
Y \ar[u]_{v}
}
$$
where $g\in\cx$. We can express $A$ as a $\kappa$-directed colimit $(h_i:A_i\to A)_{i\in I}$ of $\kappa$-presentable
objects. Let $(h_{ij}:A_i\to A_j)_{i\leq j\in I}$ denote the corresponding diagram. Since $X$ is $\kappa$-presentable, 
there is a factorization
$$
u:X \xrightarrow{\ u_i \ }  A_i \xrightarrow{\ h_i \ } A
$$
for some $i\in I$. Let $u_j=h_{ij}u_i$ for $i\leq j\in I$. Form pushouts 
$$
\xymatrix@=3pc{
A_j \ar[r]^{g_j} & B_j \\
X \ar [u]^{u_j} \ar [r]_{g} &
Y \ar[u]_{v_j}
}
$$ 
We get the diagram
$$
\xymatrix@C=3pc@R=3pc{
A \ar [r]^{f}  & B \\
A_j \ar[r]^{g_j} \ar [u]^{h_j}  & B_j \ar [u]_{k_j}\\
X \ar [r]_g \ar [u]^{u_j} & Y \ar [u]_{v_j}
}
$$ 
where the lower square and the outer rectangle are pushouts and $k_j$ is the induced morphism. Thus the upper square 
is the pushout. Hence $f$ is a $\kappa$-directed colimit in $\psh(\ck^2)$ of morphisms $g_j:A_j \to B_j$ belonging to $\Po_\kappa(\cx)$.
\end{proof} 

\begin{rem}\label{re2.6}
{
\em
For a class $\cx$ of arrows in a given cocomplete category $\ck$, and an ordinal $\lambda$, let us denote by
$\lambda$-$\sm(\cx)$ the category whose objects are the smooth chains in $\ck$ of length $\lambda$, and whose links are
in $\cx$. Recall that these are the chains $(a_{ij}:A_i \to A_j)_{i\leq j< \lambda}$ such that $(a_{ij}:A_i\to A_j)_{i<j}$
is a colimit for any limit ordinal $j<\lambda$ and $a_{i,i+1}\in\cx$ for each $i+1 <\lambda$. A morphism 
$(h_i):(a_{ij})\to (b_{ij})$ of smooth chains will be called a \textit{pushout morphism} if all squares
$$
\xymatrix@=3pc{
A_i \ar[r]^{a_{ij}} & A_j \\
B_i \ar [u]^{h_i} \ar [r]_{b_{ij}} &
B_j \ar[u]_{h_j}
}
$$
are pushouts. Note that the smoothness implies that it is sufficient to require the condition for $j=i+1$. Let
$\lambda$-$\smp(\cx)$ be the subcategory of $\lambda$-$\sm(\cx)$ with the objects as $\lambda$-$\sm(\cx)$ and
with the pushout morphisms. In particular, $1$-$\sm(\ck^2)=1$-$\smp(\ck^2)=\ck$, $2$-$\sm(\ck^2)=\ck^2$ 
and $2$-$\smp(\ck^2)=\psh(\ck^2)$. By recursion, we can generalize \ref{le2.5} to any ordinal $0<\lambda<\kappa$:

Let $\ck$ be a locally $\kappa$-presentable category, $\cx$ a set of morphisms between $\kappa$-presentable objects
and $0<\lambda<\kappa$ an ordinal. Then any chain in $\lambda$-$\smp(\Po(\cx))$ is a $\kappa$-directed colimit 
in $\lambda$-$\smp(\Po(\cx))$ of chains belonging to $\lambda$-$\smp(\Po_\kappa(\cx))$.

This statement can be used for proving the following result:

Let $\ck$ be a locally $\kappa$-presentable category and $\cx$ a set of morphisms between $\kappa$-presentable objects. 
Then
$$
\Tcsh\kappa\Po(\cx)=\Po\Tcsh\kappa\Po_\kappa(\cx).
$$ 

This result was proved in \cite{MRV} 4.20 using good colimits.
}
\end{rem}

\section{Limits of combinatorial categories}
A functor $F:\ck\to\cl$ between cellular categories will be called \textit{cellular} if it preserves colimits and cellular morphisms.
We will denote $\CAT$ the (illegitimate) 2-category of categories, functors and natural transformations and $\CELL$ the (illegitimate) 2-category 
of cellular categories, cellular functors and natural transformations. The forgetful 2-functor $U:\CELL\to\CAT$ has both a left 2-adjoint given 
by discrete cellular structures and a right 2-adjoint given by trivial ones. In particular, $U$ preserves all existing 2-limits. We are not
really interested in 2-limits in $\CELL$ but, for what follows, it is instructive to calculate pseudopullbacks.

We recall that a pseudopullback of functors $F$ and $G$ is a square in $\CAT$
$$
\xymatrix@=4pc{
\cp \ar [r]^{\bar{F}} \ar [d]_{\bar{G}}& \cl \ar [d]^{G}\\
\ck\ar [r]_{F}& \cm
}
$$
which commutes up to an isomorphism and has the 2-categorical universal property among such squares. Objects of the category $\cp$ are
triples $(K,L,t)$ where $t:FK\to GL$ is an isomorphism and morphisms $(K_1,L_1,t_1)\to (K_2,L_2,t_2)$ are pairs $(u,v)$ where $u:K_1\to K_2$,
$v:L_1\to L_2$ such that $t_2F(u)=G(v)t_1$. The functors $\bar{F}$, $\bar{G}$ are the projections and $t$'s yield the desired natural
isomorphism $F\bar{G}\to G\bar{F}$. 

Given $\cc\subseteq\ck^2$ and $\cd\subseteq\cl^2$, we get the class
$$
\Ps(\cc,\cd)=\{f\in\cp^2|\overline{F}f\in\cd,\overline{G}f\in\cc\}
$$
of morphisms in $\cp$.

\begin{lemma}\label{le3.1} $\CELL$ has pseudopullbacks.
\end{lemma}
\begin{proof}
Let $F:\ck\to\cm$ and $G:\cl\to\cm$ be cellular functors. It suffices to put
$$
\cell(\cp)=\Ps(\cell(\ck),\cell(\cl)).
$$
\end{proof} 

We will denote $\LOC$ the 2-category of locally presentable categories, colimit preserving functors and natural transformations. This
is a legitimate category which is not locally small. Recall that any colimit preserving functor between locally presentable categories has
a left adjoint. $\LOC$ has all PIE-limits, which means products, inserters and equifiers. Consequently, it has all pseudolimits, in particular 
it has pseudopullbacks. This basic result was proved in \cite{Bi} and follows from a more general limit theorem for accessible categories 
(see \cite{MP}) where one can find all needed concepts (see also \cite{AR}).  

A functor $F:\ck\to\cl$ between combinatorial categories will be called \textit{combinatorial} if it preserves colimits and cofibrations.
$\COMB$ will denote the 2-category of combinatorial categories, combinatorial functors and natural transformations. Again, this category
is legitimate but not locally small and the forgetful 2-functor $V:\COMB\to\LOC$ has both a left 2-adjoint and right 2-adjoint given 
by discrete and trivial combinatorial structures. Thus $V$ preserves all existing 2-limits.  Moreover $\COMB$ is a full sub-2-category
of $\CELL$.

\begin{theo}\label{th3.2} $\COMB$ has pseudopullbacks calculated in $\CELL$.
\end{theo}
\begin{proof}
Consider a pseudopullback in $\CELL$
$$
\xymatrix@=4pc{
\cp \ar [r]^{\bar{F}} \ar [d]_{\bar{G}}& \cl \ar [d]^{G}\\
\ck\ar [r]_{F}& \cm
}
$$
where $F$ and $G$ are combinatorial functors. We have to show that $\cp$ is combinatorial, i.e., that $\Ps(\cof(\ck),\cof(\cl))$ is
cofibrantly generated by a set of morphisms.

There is an uncountable regular cardinal $\kappa$ such that the categories $\ck,\cl,\cm$ are locally $\kappa$-presentable, both $\cof(\ck)$
and $\cof(\cl)$ are cofibrantly generated by morphisms between $\kappa$-presentable objects, $\cp$ is locally $\kappa$-presentable and the functors 
$F,G,\overline{F},\overline{G}$ preserve $\kappa$-filtered colimits and $\kappa$-presentable objects. Following \cite{L} A.1.5.12, 
both $\cof(\ck)$ and $\cof(\cl)$ are cellularly generated by morphisms between $\kappa$-presentable objects. Let $\cx$ consist of morphisms $f$ between 
$\kappa$-presentable objects in $\cp$ such that $\overline{F}f\in\cof(\cl)$ and $\overline{G}f\in\cof(\ck)$. 
We will prove that
$$
\cell(\cx)=\Ps(\cof(\ck),\cof(\cl)),
$$
which proves the theorem. Let $\cof_\kappa(\ck)$ denote cofibrations between $\kappa$-presentable objects in $\ck$ and the same for $\cof_\kappa(\cl)$.
Then 
$$
\cx=\Ps(\cof_\kappa(\ck),\cof_\kappa(\cl)).
$$ 
Thus it suffices to prove the equations   
\begin{align*}
\Ps(\cof(\ck),\cof(\cl))&=\Tc\Ps(\Po\cof_\kappa(\ck),\Po\cof_\kappa(\cl))\\
&=\Tc\Po\Ps(\cof_\kappa(\ck),\cof_\kappa(\cl)).
\end{align*}

In the first equation, the right-hand side is obviously included in the left-hand side. Let $e$ belong to the left-hand side of the first equation.
Following \cite{MRV} 4.11, there are $\kappa$-good $\kappa$-directed diagrams $D:P\to\ck$ and $E:Q\to\cl$ with links in $\Po\cof_\kappa(\ck)$
and $\Po\cof_\kappa(\cl)$ resp. such that $\overline{F}e$ is the composite of $E$ and $\overline{G}e$ is the composite 
of $D$. Thus there are isomorphisms $u$ and $v$ in $\cm$ such that the square
$$
\xymatrix@=4pc{
FD\perp \ar [r]^{F\overline{G}e} \ar [d]_{u}& F\colim D \ar [d]^{v}\\
GE\perp\ar [r]_{G\overline{F}e}& G\colim E
}
$$
commutes. In what follows, $\delta:D\to\colim D$ and $\varepsilon:E\to\colim E$ are colimit cocones. We can assume that neither $P$ nor $Q$ have
the greatest element.

Let $\tilde{P}$ denote the set of all non-empty initial segments $X$ of $P$. 
For each $X\in\tilde{P}$, we get the induced morphism  $\delta_X:\colim D_X\to\colim D$ where $D_X$ denotes the restriction of $D$ on $X$. 
Analogously, we have $\varepsilon_Y:\colim E_Y\to\colim E$ for $Y\in\tilde{Q}$.
Given $X\in\tilde{P}$ and $x\in P$, let $X(x)=X\cup\downarrow x$. We are going to show that $X\neq P$ implies $X(x)\neq P$. Since $P$ is directed
and does not have the greatest element, there is $x<y\in P$. Clearly, either $y\notin X(x)$ or $X(x)=X$.
 
By recursion on all ordinals $i$ and $j$, $i<j$, we will construct smooth chains $X_i\subseteq X_j$ in $\tilde{P}$, $Y_i\subseteq Y_j$ in $\tilde{Q}$ and 
$$
(k_{ij},l_{ij}):(K_i,L_i,u_i)\to (K_j,L_j,u_j)
$$ 
in $\cp$ such that each $K_i$ is a colimit of the restriction of $D$ on an initial segment $X_i\in\tilde{P}$,
each $L_i$ is a colimit of the restriction of $E$ on an initial segment $Y_i\in\tilde{Q}$, $k_{ij}:K_i\to L_i$, $l_{ij}:L_i\to L_j$
are the induced morhisms, and, for the induced morphisms $k_i:K_i\to\colim D$, $l_i:L_i\to\colim E$, the square
$$
\xymatrix@=4pc{
FK_i \ar [r]^{Fk_i} \ar [d]_{u_i}& F\colim D \ar [d]^{v}\\
GL_i\ar [r]_{Gl_i}& G\colim E
}
$$
commutes. The construction will terminate at the ordinal $k$ when both $X_k=P$ and $Y_k=Q$ become true. It follows easily from \cite{MRV} 4.19
that each morphism $k_{ij}$ is in $\Po\cof_\kappa(\ck)$, and similarly $l_{ij}\in\Po\cof_\kappa(\cl)$, thus the smooth chain  
$(k_{ij},l_{ij})_{i\leq j<k}$ has links in $\Ps(\Po\cof_\kappa(\ck),\Po\cof_\kappa(\cl))$. In this way we get 
that $e$ belongs to the right-hand side of the first equation.

We put $K_0=D\perp$, $L_0=E\perp$ and $u_0=u$. Let us have $(K_i,L_i,u_i)$. If $X_i=P$ and $Y_i=Q$, we are finished, and we put $k=i$. Otherwise,
either $Y_i\neq Q$, or $X_i\neq P$. Assume the first case, the second is handled symmetrically. Choose $y_1\in Q-Y_i$
and put $Y_{i1}=Y_i(y_1)$, $L_{i1}=\colim E_{Y_{i1}}$. Following \cite{MRV} 4.19, 
the induced morphism $s_0:L_i\to L_{i1}$ belongs to $\Po\cof_\kappa(\cl)$. Let
$$
\xymatrix@=3pc{
L_i \ar[r]^{s_0} & L_{i1} \\
A \ar [u]^{a} \ar [r]_{h} &
B \ar[u]_{b}
}
$$
be a corresponding pushout with $h\in\cof_\kappa(\cl)$. 
Since $GB$ is $\kappa$-presentable and $P$ $\kappa$-directed, there is $x_0\in P$ and 
$$
f:GB\to FDx_0
$$ 
such that $v^{-1}G(\varepsilon_{Y_{i1}})G(b)=F(\delta_{x_0})f$. We obtain the morphisms
$$
f_1:GB \xrightarrow{\ f \ } FDx_0 \xrightarrow{\  \ } F\colim D_{X_i(x_0)}
$$
and
$$
f_2:GL_i \xrightarrow{\ u_i^{-1} \ } FK_i \xrightarrow{\  \ } F\colim D_{X_i(x_0)}
$$
Since
$$
F\delta_{X_i(x_0)}f_1G(h)=F\delta_{X_i(x_0)}f_2G(a),
$$
$F\delta_{X(x)}:F\colim D_{X(x)}\to F\colim D$ is a $\kappa$-directed colimit cocone and $GA$ is $\kappa$-presentable, there is $x_0\leq x_1\in P$
such that the morphisms
$$
g_1:GB \xrightarrow{\ f_1 \ }  F\colim D_{X_i(x_0)} \xrightarrow{\  \ } F\colim D_{X_i(x_1)}
$$
and
$$
g_2:GL_i \xrightarrow{\ g_1 \ } F\colim D_{X_i(x_0)} \xrightarrow{\  \ } F\colim D_{X_i(x_1)}
$$
satisfy $g_1G(h)=g_2G(a)$. We put $X_{i1}=X(x_1)$, $K_{i1}=\colim D_{X_{i1}}$ and $u_{i1}:GL_{i1}\to FK_{i1}$ is the induced morphism
from the pushout defining $GL_{i1}$. Following \cite{MRV} 4.19, the induced morphism $t_0:K_i\to K_{i1}$ belongs to $\Po\cof_\kappa(\ck)$. 
We have $u_{i1}G(s_0)u_i=F(t_0)$ and $v^{-1}G(\varepsilon_{Y_{i1}})=F(\delta_{X_{i1}})u_{i1}$. 

Now, in the same way as above, we get initial segments $X_{i2}\in\tilde{P}$ and $Y_{i2}\in\tilde{Q}$, the objects $K_{i2}=\colim D_{X_{i2}}$ 
and $L_{i2}=\colim E_{Y_{i2}}$ and morphisms $t_1:K_{i1}\to K_{i2}$ in $\Po\cof_\kappa(\ck)$, $s_1:L_{i1}\to L_{i2}$ in $\Po\cof_\kappa(\cl)$ and 
$u_{i2}:FK_{i2}\to GL_{i2}$ such that $vF(\delta_{X_{i2}})=G(\varepsilon_{Y_{i2}})u_{i2}$ and $u_{i2}F(t_1)u_{i1}=G(s_1)$. We have
$$
u_{i2}F(t_1t_0)=u_{i2}F(t_1)u_{i1}G(s_0)u_i=G(s_1s_0)u_i.
$$
Continuing this procedure, we get morphisms $u_{in}$ with alternating directions, whose squares with $v$ or $v^{-1}$ commute and all squares
between $u_{in}$'s for odd $n$  and between $u_{in}$'s  for even $n$ commute. We put $X_{i+1}=\cup X_{in}$, $Y_{i+1}=\cup Y_{in}$, 
$K_{i+1}=\colim D_{X_{i+1}}$, $L_{i+1}=\colim E_{Y_{i+1}}$ and $u_{i+1}=\colim u_{i(2n)}$. Clearly, $u_{i+1}$ is an isomorphism; its inverse 
is $\colim u_{i(2n+1)}$.

 

The construction of the items at stage $i$ for $i$ a limit ordinal is dictated by the smoothness requirements. Clearly, there is an ordinal $k$
where the construction stops: $X_k=P$ and $Y_k=Q$. Since now $f_k$ is the identity on $\colim D$, and similarly for $g_k$, it follows that
$u_k=v$. Thus $e$ is the composite of the diagram $(k_{ij},l_{ij})_{i<k\leq k}$ as desired.

The second equation is the consequence of
$$
\Ps(\Po\cof_\kappa(\ck),\Po\cof_\kappa(\cl))=\Po\Ps(\cof_\kappa(\ck),\cof_\kappa(\cl)).
$$
In this equation, the right-hand side is obviously contained in the left-hand side. Let $e$ belong to the left-hand side.
Then $\overline{F}e:L_1\to L_2$ is a pushout of a morphism from $\cof_\kappa(\cl)$ and $\overline{G}e:K_1\to K_2$ is is 
a pushout of a morphism from $\cof_\kappa(\ck)$. There are isomorphisms $u$ and $v$ in $\cm$ such that the square
$$
\xymatrix@=4pc{
FK_1 \ar [r]^{F\overline{G}e} \ar [d]_{u}& FK_2 \ar [d]^{v}\\
GL_1\ar [r]_{G\overline{F}e}& GL_2
}
$$
commutes. 

Following \ref{le2.5}, $\overline{G}e$ is a $\kappa$-directed colimit in $\psh(\ck^2)$ of morphisms from $\cof_\kappa(\ck)$ 
and $\overline{F}e$ is a $\kappa$-directed colimit in $\psh(\cl^2)$ of morphisms from $\cof_\kappa(\cl)$. Let
$g_0:L_{10}\to L_{20}$ be a morphism from $\cof_\kappa(\cl)$, $(l_{10},l_{20}):g_0\to\overline{F}e$ a morphism
in $\psh(\cl^2)$, $f_0:K_{10}\to K_{20}$ a morphism from $\cof_\kappa(\ck)$ and $(k_{10},k_{20}):f_0\to\overline{G}e$ 
a morphism in $\psh(\ck^2)$. There is a morphism $g_1:L_{11}\to L_{21}$ from $\cof_\kappa(\cl)$ and a morphism
$(l_{11},l_{21}):g_1\to\overline{F}e$ with factorizations 
$$
(l_{10},l_{20})=(l_{11},l_{21})\cdot(l_{101},l_{201})
$$ 
and
$$
(u,v)\cdot(Fk_{10},Fk_{20})=(Gl_{11},Gl_{21})(u_1,v_1).
$$
There is a morphism $f_1:K_{11}\to K_{21}$ from $\cof_\kappa(\ck)$ and a morphism
$(k_{11},k_{21}):f_1\to\overline{G}e$ with factorizations 
$$
(k_{10},k_{20})=(k_{11},k_{21})\cdot(k_{101},k_{201})
$$ 
and
$$
(u^{-1},v^{-1})\cdot(Gl_{11},Gl_{21})=(Fk_{11},Fk_{21})(u_2,v_2).
$$
We continue this procedure and take colimits of the resulting chains $g=\colim g_n$ and $f=\colim f_n$. Since
$Gg$ and $Ff$ are isomorphic, we get a morphism $e'$ from $\Ps(\cof_\kappa(\ck),\cof_\kappa(\cl))$ and a pushout
morphism $e'\to e$. Therefore $e$ belongs to $\Po\Ps(\cof_\kappa(\ck),\cof_\kappa(\cl))$.
\end{proof}

\begin{coro}\label{cor3.3} $\COMB$ has PIE-limits calculated in $\CELL$.
\end{coro}
\begin{proof}
Products of combinatorial categories are evident:
$$
\prod\limits_{i\in I} (\ck_i,\cx_i)=(\prod\limits_{i\in I}\ck_i,\prod\limits_{i\in I}\cx_i).
$$

Let $F,G:\ck\to\cl$ be combinatorial functors, $\varphi,\psi:F\to G$ na\-tu\-ral transformations and $\Eq(\varphi,\psi)$ their equifier
in $\CAT$. Consider a pseudopullback
$$
\xymatrix@=4pc{
\cp \ar [r]^{} \ar [d]_{}& \ck \ar [d]^{\Id}\\
\Eq(\varphi,\psi)_t\ar [r]_{V}& \ck_t
}
$$
where $V:\Eq(\varphi,\psi)\to\ck$ is the full embedding. Then $\cp$ is an equifier of $\varphi$ and $\psi$ in $\COMB$.

Finally, let $F,G:\ck\to\cl$ be combinatorial functors and $\Ins(F,G)$ their inserter in $\CAT$. Consider a pseudopullback
$$
\xymatrix@=4pc{
\cp \ar [r]^{} \ar [d]_{}& \ck \ar [d]^{\Id}\\
\Ins(F,G)_t\ar [r]_{V}& \ck_t
}
$$
where $V:\Ins(F,G)\to\ck$ is the forgetful functor. Then $\cp$ is an inserter of $F$ and $G$ in $\COMB$.

Clearly, all PIE-limits above are calculated in $\CELL$. 
\end{proof}

The consequence is that $\COMB$ has pseudolimits and lax limits (calculated in $\CAT$). The same is true for $\CELL$. Another consequence 
is \cite{L} 2.8.3 (see \cite{M} as well).

\begin{coro}\label{cor3.4} Let $\ck$ be a combinatorial category and $\cc$ a small ca\-te\-go\-ry. Then the functor category $\ck^\cc$
is combinatorial with respect to the pointwise combinatorial structure.
\end{coro}
\begin{proof}
The cotensor $[\cc,\ck]$ taken in $\CELL$ is the cellular category described in the Corollary. Since $[\cc,\ck]$ can be constructed using
PIE-limits, \ref{cor3.3} implies that $[\cc,\ck]$ is in $\COMB$ provided that $\ck$ is in $\COMB$.
\end{proof} 

\begin{rem}\label{re3.5}
{
\em
Let $\ck$ be a Grothendieck abelian category and $\cc$ a class of objects in $\ck$. A $\cc$-monomorphism is a monomorphism whose cokernel belongs
to $\cc$. The class $\cc$-$\Mono$ of these monomorphisms makes $\ck$ a cellular category. Cellular objects here are precisely $\cc$-filtered objects, 
i.e., objects $K$ such that the morphism $0\to K$ is a transfinite composite of $\cc$-monomorphisms. A class $\cc$ is deconstructible if it is 
a class of $\cs$-filtered objects for a set $\cs$. The fundamental fact (basically due to \cite{SS}) is that $\cc$ is deconstructible if and
only if the cellular category $(\ck,\cc$-$\Mono)$ is combinatorial. 

Let $\cc$ be deconstructible and $\C(\ck)$ be the category of complexes over $\ck$. Since it is a functor category, \ref{cor3.4} implies that $\C(\ck)$ 
is combinatorial with respect to pointwise $\cc$-monomorhisms. Consequently the class $\C(\cc)$ of complexes with components in $\cc$ is deconstructible,
which was proved in \cite{S} 4.2 (1) using generalized Hill lemma. 
}
\end{rem}

Let $T:\ck\to\ck$ be a colimit preserving monad on a combinatorial category $\ck$ and $U:\Alg(T)\to\ck$ the forgetful functor from the category
of $T$-algebras. Then $\Alg(T)$ is locally presentable (see \cite{AR} Remark 2.78) and $U$ preserves colimits (see \cite{Bo} 4.3.2).
Thus $\Alg(T)$ is a combinatorial category where $f$ is a cofibration if and only if $Uf$ is a cofibration.

\begin{coro}\label{cor3.6} Let $T:\ck\to\ck$ be a colimit preserving monad on a combinatorial category $\ck$. Then $\Alg(T)$ is combinatorial.
\end{coro}
\begin{proof}
The combinatorial category $\Alg(T)$ is given by a pseudopullback
$$
\xymatrix@=4pc{
\Alg(T) \ar [r]^{U} \ar [d]_{\Id}& \ck \ar [d]^{\Id}\\
\Alg(T)_t \ar [r]_{U}& \ck_t
}
$$
\end{proof} 

\begin{rem}\label{re3.7}
{
\em
Let $T:\ck\to\ck$ be a colimit preserving monad on a Grothendieck abelian category and $\cc$ a deconstructible class of objects in $\ck$. 
Then the class of $T$-algebras $A$ with $UA\in\cc$ is deconstructible in $T$-$\Alg$. This result was proved in \cite{Be} A.7 and follows
from \ref{cor3.6} and \ref{re3.5}.
}
\end{rem}

\begin{rem}\label{re3.8}
{
\em
More generally, let $F:\ck\to\cl$ be a colimit preserving functor from a locally presentable category $\ck$ to a combinatorial category $\cl$.
In the same way as above, we get a combinatorial structure on $\ck$ where $f$ is a cofibration if and only if $Ff$ is a cofibration. We have
a pseudopullback
$$
\xymatrix@=4pc{
\ck \ar [r]^{F} \ar [d]_{\Id}& \cl \ar [d]^{\Id}\\
\ck_t\ar [r]_{F}& \cl_t
}
$$
In accordance with \cite{HS} 4.1, we call this combinatorial structure \textit{left-induced} from $\cl$. 

Let us observe that both equifiers and inserters are given as left-induced structures. Since PIE-limits yield all pseudolimits, we could prove only the special case of \ref{th3.2} giving the existence of left-induced structures.  
}
\end{rem}

\section{Limits of combinatorial model categories}
Any combinatorial model category $\ck$ has two underlying combinatorial categories $W_1(\ck)=(\ck,\cc)$ and $W_2(\ck)=(\ck,\cc_0)$
where $\cc$ is the class of cofibrations of $\ck$ and $\cc_0$ is the class of trivial cofibrations. On every locally presentable
category $\ck$ there is a \textit{trivial} combinatorial model structure $\ck_{tm}$ such that both $W_1(\ck_{tm})$ and $W_2(\ck_{tm})$
are trivial combinatorial categories and a \textit{discrete} combinatorial model structure $\ck_{dm}$ such that both $W_1(\ck_{dm})$ 
and $W_2(\ck_{dm})$ are discrete combinatorial categories. Weak equivalences are all morphisms in the both cases. More generally, any combinatorial 
category $\ck$ yields a combinatorial model category $m(\ck)$ such that $W_i(m(\ck))=\ck$ for $i=1,2$. Again, any morphism of $\ck$ is a weak 
equivalence in $m(\ck)$. In particular, $\ck_{tm}=m(\ck_t)$ and $\ck_{dm}=m(\ck_d)$ .

Let $\CMOD$ denote the category of combinatorial model categories and left Quillen functors. We get the functors 
$$
W_1,W_2:\CMOD\to\COMB.
$$

\begin{lemma}\label{le4.1} $W_2$ preserves pseudopullbacks existing in $\CMOD$.
\end{lemma}
\begin{proof} It follows from the fact that $m(-):\COMB\to\CMOD$ is left adjoint to $W_2$.
\end{proof}
 
We know that
$$
\Ps(\tcof(\ck),\tcof(\cl))\subseteq\cof(\cp)\subseteq\Ps(\cof(\ck),\cof(\cl)).
$$
We will show that $W_1$ does not need to preserve existing pseudopullbacks.

\begin{exam}\label{ex4.2}
{
\em
Let $\ck$ be the standard model category of simplicial sets. Let $t:0\to 1$ and $\cl$ be the model structure on simplicial sets
where $\cof(\{t\})$ is the class of cofibrations and any morphism is a weak equivalence. It is easy to see that cofibrations
are precisely coproduct injections $K\to K\coprod D$ where $K$ is a simplicial set and $D$ is a discrete simplicial set. We will show
that the discrete model structure $\ck_d$ on simplicial sets yields a pseudopullback
$$
\xymatrix@=4pc{
\ck_d \ar [r]^{\Id} \ar [d]_{\Id}& \cl \ar [d]^{\Id}\\
\ck\ar [r]_{\Id}& \ck_{tm}
}
$$
in $\COMB$. Consider a model structure $\cp$ on simplicial sets such that $\Id:\cp\to\ck$ and $\Id:\cp\to\cl$ are left Quillen functors.
Since $\cof(\{t\})$ is the class of trivial cofibrations in $\cl$, the intersection of trivial cofibrations in $\ck$ and $\cl$ contains 
only isomorphisms. Thus $\tcof(\cp)=\Iso$. Since $\cof(\{t\})$ is the intersection of cofibrations in $\ck$ and $\cl$, trivial fibrations 
in $\cp$ should contain $\cof(\{t\})^\square$, which is the class of surjective simplicial maps. Since trivial cofibrations in $\cp$ are isomorphisms, 
$\cof(\cp)^\square$ is the class of weak equivalences in $\cp$. Hence $\cof(\cp)^\square$ has the 2-out-of-3 property and thus it contains all coproduct 
injections $u:A\to A\coprod B$. The reason that there is always $f:A\coprod B\to B$ with $fu=\id_A$. Since $\cof(\cp)\cap\cof(\cp)^\square=\Iso$
(see \cite{AHRT} III.4 (2)), $\cof(\cp)=\Iso$.   
}
\end{exam}

We do not know whether $\CMOD$ has pseudopullbacks. 

\begin{rem}\label{re4.3}
{
\em
(1) The existence of pseudopullbacks would imply the existence of PIE-limits in $\CMOD$. Since products of combinatorial model categories
exist and are preserved by $W_1$ and $W_2$, PIE-limits can be obtained from pseudopullbacks of the kind
$$
\xymatrix@=4pc{
\ck \ar [r]^{\Id} \ar [d]_{F}& \ck_{tm} \ar [d]^{F}\\
\cl\ar [r]_{\Id}& \cl_{tm}
}
$$
like in \ref{cor3.4}. Following \ref{le4.1}, $f$ is a trivial cofibration in $\ck$ if and only if $Ff$ is a trivial cofibration. If the same holds 
for cofibrations then $\ck$ is left-induced in the sense of \cite{HS}. 

(2) In particular, the existence of pseudopullbacks would imply the existence of lax limits in $\CMOD$. But Barwick proved that they always
exist and are preserved not only by $W_2$ but also by $W_1$ (\cite{Ba} 2.30).
}
\end{rem}

\end{document}